\documentclass{amsart}

\usepackage{amssymb,amsthm}

\theoremstyle{definition}
\newtheorem{thm}{Theorem}[section]
\newtheorem*{thm*}{Theorem}
\newtheorem{exmpl}[thm]{Example}
\newtheorem{cor}[thm]{Corollary}

\newcommand{\id}{{\rm id}}
\newcommand{\fs}{{\mathfrak S}}

\numberwithin{equation}{section}

\title[Segments and Genocchi numbers]{Cardinality of $\ell_1$-Segments and Genocchi Numbers}

\author[C. Zara]{Catalin Zara}
\address{Department of Mathematics, University of Massachusetts
Boston, MA 02125}
\email{catalin.zara@umb.edu}

\date{April 18, 2013}

\begin{document}

\begin{abstract} 
We prove that the Genocchi numbers of first and second kind give the cardinality  of certain segments in permutation spaces $\fs_n$ with respect to 
the $\ell_1$-distance. Experimental data suggests that those segments have
maximal cardinality among all segments in the corresponding spaces.
\end{abstract}

\maketitle

\tableofcontents

\section{Introduction}

For a positive integer $n$ let $\fs_n$ be the group of permutations of $[n]=\{1, 2,\ldots, n\}$. We define a distance $D$ on $\fs_n$ by
\begin{equation*}
D(u,v) = \sum_{i=1}^n |u(i)-v(i)|\; .
\end{equation*}
This distance is right invariant: $D(uw, vw) = D(u,v)$ for every $u,v,w \in \fs_n$. For $u \in \fs_n$, the segment $[\id, u]$ is the set
\begin{equation*}
[\id, u] = \{ v \in \fs_n \mid D(\id, v) + D(v, u) = D(\id , u)\}\; .
\end{equation*}
Then $v \in [\id, u]$ if and only if
\begin{equation}\label{eq:segment_cond}
\min(i,u(i)) \leqslant v(i) \leqslant \max(i, u(i))
\end{equation}
for all $i=1,2,\ldots, n$. Experimental data suggests that the maximal 
cardinality of a segment  $[\id, u]$ in $\fs_n$ is attained for
\begin{equation*}
u = w_n = 
\begin{cases}
(m\!+\!1 \, m\!+\!2\; \ldots 2m \; 1 \; 2\; \ldots m) \; ,  \text{ if } n=2m\\
(m\!+\!1 \, m\!+\!2\; \ldots 2m\!+\!1 \; 1 \; 2\; \ldots m)\; , \text{ if } n=2m\!+\! 1
\end{cases}\; .
\end{equation*}
For example, $w_1=(1)$, $w_2 = (21)$, $w_3 = (231)$, 
$w_4= (3412)$, $w_5= (34512)$, \ldots .

The sequence $\# [\id, w_n]$ starts with 1, 2, 3, 8, 17, 56, 155, 608, 2073, 9440, ...  
and a search on the Online Encyclopedia of Integer Sequences (at oeis.org) shows 
a match with the sequence ``A099960: An interleaving of the Genocchi numbers 
of the first and second kind,'' from the fourth term on. 
The goal of this article is to give a bijective proof of this 
characterization of $\#[\id, w_n]$, the cardinality of the segment $[id, w_n]$.

\section{Genocchi Numbers}

The Genocchi numbers have combinatorial descriptions in terms of Dumont permutations. The Genocchi number of the first kind $G_{2n+2}$ is the cardinality of the set $B_{2n}$ of permutations $\pi \in \fs_{2n}$ such that 
\begin{equation}\label{eq:gen1_cond}
\pi(2i) \leqslant 2i \leqslant \pi(2i\!-\!1) 
\end{equation}
for all $i=1, \ldots, n$.  (\cite{dum-bar81}, \cite{dum74}). Note that $\pi(2n\!-\!1) = 2n$ for all $\pi \in B_{2n}$. For example, $G_6=3$ is the cardinality of 
\begin{equation*}
B_4=\{ (2143), (3142), (3241)\}\; .
\end{equation*}

The Genocchi median (or number of the second kind)  $H_{2n+1}$ is the cardinality of the set $C_{2n}$ of permutations $\pi \in \fs_{2n}$ such that 
\begin{equation}\label{eq:gen2_cond} 
\pi(2i) < 2i \leqslant \pi(2i\!-\!1)
\end{equation}
for  $i=1, \ldots, n$.  (\cite{dum-rad94}). Note that $\pi(2)=1$ and $\pi(2n\!-\!1) = 2n$ 
for all $\pi \in C_{2n}$. For example, $H_7=8$ is the cardinality of 
\begin{align*}
C_6 = \{ & (415263), (315263), (314265), (514263), \\
			& (215364), (214365), (415362), (514362)\}
\end{align*}

We prove that $\#[\id, w_{2m}] = H_{2m+3}$ and $\# [\id, w_{2m+1}] = G_{2m+4}$ 
by establishing explicit bijections between $[\id, w_{2m}]$ and $C_{2m+2}$ and 
between $[\id, w_{2m+1}]$ and $B_{2m+2}$.

\section{Cardinality of Segments}

The main results and their proofs are similar. 
We start with the case $n=2m\!+\!1$.

Define $\rho\colon \fs_{2m+1} \to \fs_{2m+2}$ by 
\begin{equation*}
\rho(u) = (u(1) \; u(2) \; \ldots \; u(2m) \; 2m\!+\!2 \; u(2m\!+\!1))
\end{equation*}
In other words, $\rho$ inserts the value $2m+2$ between on 
position $2m+1$ and shifts the last position to the right by one space. For example, 
if $m=1$ and $u=(231) \in \fs_3$, then $\rho(u) = (2341) \in \fs_4$.

\begin{thm}
Define the function  $g\colon \fs_{2m+1} \to \fs_{2m+2}$ by
$g(u) = \rho(\alpha u^{-1}\beta^{-1})\; ,$
where $\alpha = (1 \, 3\, \ldots\; 2m\!+\!1\; 2\, 4\, \ldots\, 2m)$ and 
$\beta = (2 \, 4\, \ldots\, 2m \, 2m\!+\!1\, 1\, 3 \,  \ldots \, 2m\!-\!1) $. 
Then 
\begin{equation*}
u \in [\id, w_{2m+1}] \Longleftrightarrow g(u) \in B_{2m+2}\; .
\end{equation*}
\end{thm}

\begin{proof}
Let $u \in \fs_{2m+1}$. Then 
\begin{align*}
g(u)&(2m\!+\!1) = 2m\!+\!2 \\
g(u)&(2m\!+\!2) = \alpha u^{-1}\beta^{-1}(2m\!+\!1) \leqslant 2m\!+\!1 \; ,
\end{align*}
hence conditions \eqref{eq:gen1_cond} are satisfied by all $g(u)$ for $i=m\!+\!1$. The \emph{value} $m\!+\!1$ can occur on any position of a permutation $u \in [\id, w_{2m+1}]$.

Let $1 \leqslant i \leqslant m$. Then
\begin{equation*}
g(u)(2i)  = \rho(\alpha u^{-1}\beta^{-1})(2i) = \alpha(u^{-1}(\beta^{-1}(2i))) = \alpha(u^{-1}(i))\; .
\end{equation*}
Then the condition $g(u)(2i) \leqslant 2i$ is equivalent to 
\begin{equation*}
u^{-1}(i) \in [1, i] \cup [m\!+\!2, m\!+\!i\!+\!1]\; ,
\end{equation*}
and that is the condition for the \emph{value} $i$ to occur in $u$ on a position that satisfies condition \eqref{eq:segment_cond} for $\id$ and $w_{2m+1}$.
Similarly,
\begin{equation*}
g(u)(2i\!-\!1)  = \rho(\alpha u^{-1}\beta^{-1})(2i\!-\!1) = \alpha(u^{-1}(\beta^{-1}(2i\!-\!1))) = \alpha(u^{-1}(i\!+\!m\!+\!1))\; .
\end{equation*}
Then the condition $g(u)(2i\!-\!1) \geqslant 2i$ is equivalent to 
\begin{equation*}
u^{-1}(i\!+\!m\!+\!1) \in [i, m\!+\!1] \cup [m\!+\!i\!+\!1, 2m\!+\!1]\; ,
\end{equation*}
and that is the condition for the \emph{value} $i\!+\!m\!+\!1$ to occur in $u$ on a position that satisfies condition \eqref{eq:segment_cond} for $\id$ and $w_{2m+1}$.

To summarize: $g(u)$ satisfies conditions \eqref{eq:gen1_cond} for positions $1, 2, \ldots, 2m$ if and only if $u$ satisfies \eqref{eq:segment_cond} for values $1, 2, \ldots,  m, m\!+\!2, \ldots, 2m$. Since $g(u)$ always satisfies \eqref{eq:gen1_cond} for positions $2m\!+\!1, 2m\!+\!2$ and $u$ always satisfies \eqref{eq:segment_cond} for the value $m\!+\!1$, the equivalence is established and the proof is complete.
\end{proof}

\begin{cor}
The cardinality of the $\ell_1$-segment $[\id, w_{2m+1}]$ is  $G_{2m+4}$.
\end{cor}

\begin{proof}
The function $g$ induces a bijection from $[\id, w_{2m+1}]$ to $B_{2m+2}$.
\end{proof}

\begin{exmpl}
For $m=1$ we have $w_3 = (231)$ and 
$$[(123), (231)] = \{ (123), (132), (231) \}$$
and
$$B_4=\{ (2143), (3142), (3241)\}\; .$$
The function $u \to \rho((132)u^{-1}(312))$ sends $(123)$ to $(2143)$, $(132)$ to $(3142)$, and $(231)$ to $(3241)$, establishing a bijection between $[(123),(231)]$ and $B_4$.
\end{exmpl}


The case $n=2m$ is very similar. Define 
$\eta \colon \fs_{2m} \to \fs_{2m+2}$ by
\begin{equation*}
\eta(u) = (u(1)\!+\!1 \; 1\; u(2)\!+\!1 \; \ldots \; u(2m\!-\!1)\!+\!1 \; 2m\!+\!2\; u(2m)\!+\!1)
\end{equation*}
Explicitly, $\eta$ inserts the value 1 on the second position, $2m\!+\!2$ on position 
$2m\!+\!1$, increases the other values by 1 and shifts them to fill the remaining positions. For example, $\eta((2413)) = (315264)$.

\begin{thm}\label{thm:max_even}
Define the function $h\colon \fs_{2m} \to \fs_{2m+2}$ by
$h(u) = \eta(\alpha u^{-1}\beta^{-1})\; $
where
$\alpha = (1 \, 3\, \ldots\; 2m\!-\!1\; 2\, 4\, \ldots\, 2m)$,  
$\beta = (3 \, 5\, \ldots \, 2m\!-\!1\, 1\, 2m\, 2 \, 4\, \ldots\, 2m\!-\!2)$. 
Then 
\begin{equation*}
u \in [\id, w_{2m}] \Longleftrightarrow h(u) \in C_{2m+2}\; .
\end{equation*}
\end{thm}

\begin{proof}
Let $u \in \fs_{2m}$. Then 
\begin{align*}
h(u)&(1) = \alpha u^{-1}\beta^{-1}(1) + 1 = \alpha(u^{-1}(m))+1 \geqslant 2\\
h(u)&(2) = 1 < 2\\
h(u)&(2m\!+\!1) = 2m\!+\!2 \geqslant 2(m\!+\!1)\\
h(u)&(2m\!+\!2) = \alpha u^{-1}\beta^{-1}(2m) +1 = \alpha(u^{-1}(m\!+\!1))\!+\!1 < 2m\!+\!2 \; ,
\end{align*}
hence conditions \eqref{eq:gen2_cond} are satisfied by all $h(u)$ for $i=1$ and $i=m\!+\!1$. The \emph{values} $m$ and $m\!+\!1$ can occur on any position of a permutation $u \in [\id, w_{2m}]$.

Let $2 \leqslant i \leqslant m$. Then
\begin{equation*}
h(u)(2i)  = \eta(\alpha u^{-1}\beta^{-1})(2i) = \alpha(u^{-1}(\beta^{-1}(2i\!-\!1)))+1 = \alpha(u^{-1}(i\!-\!1))+1\; .
\end{equation*}
Then the condition $h(u)(2i)<2i$ is equivalent to 
\begin{equation*}
u^{-1}(i\!-\!1) \in [1, i\!-\!1] \cup [m\!+\!1, m\!+\!i\!-\!1]\; ,
\end{equation*}
and that is the condition for the \emph{value} $i\!-\!1$ to occur in $u$ on a position that satisfies condition \eqref{eq:segment_cond} for $\id$ and $w_{2m}$.
Similarly,
\begin{equation*}
h(u)(2i\!-\!1)  = \eta(\alpha u^{-1}\beta^{-1})(2i\!-\!1) = \alpha(u^{-1}(\beta^{-1}(2i\!-\!2)))+1 = \alpha(u^{-1}(i\!+\!m))+1\; .
\end{equation*}
Then the condition $h(u)(2i) \geqslant 2i$ is equivalent to 
\begin{equation*}
u^{-1}(i\!+\!m) \in [i, m] \cup [m\!+\!i, 2m]\; ,
\end{equation*}
and that is the condition for the \emph{value} $i+m$ to occur in $u$ on a position that satisfies condition \eqref{eq:segment_cond} for $\id$ and $w_{2m}$.

To summarize: $h(u)$ satisfies conditions \eqref{eq:gen2_cond} for positions $3, 4, \ldots, 2m$ if and only if $u$ satisfies \eqref{eq:segment_cond} for values $1, 2, \ldots, m\!-\!1, m\!+\!2, \ldots, 2m$. Since $h(u)$ always satisfies \eqref{eq:gen2_cond} for positions $1, 2, 2m\!+\!1, 2m\!+\!2$ and $u$ always satisfies \eqref{eq:segment_cond} for values $m, m\!+\!1$, the equivalence is established and the proof is complete.
\end{proof}

\begin{cor}
The cardinality of the $\ell_1$-segment $[\id, w_{2m}]$ is $H_{2m+3}$.
\end{cor}

\begin{proof}
The function $h$ induces a bijection from $[\id, w_{2m}]$ to $C_{2m+2}$.
\end{proof}

\begin{exmpl}
For $m=2$ we have $w_4=(3412)$ and
\begin{align*}
[(1234), (3412)] = \{ & (1234), (1324), (1423), (1432), \\
                                & (2314), (2413), (3214), (3412) \}  \; .
\end{align*}
The function $u \to \eta((1324)u^{-1}(2413))$ sends the 
permutations in $[(1234), (3412)]$ to the corresponding permutations in 
\begin{align*}
C_6 = \{ & (415263), (315263), (314265), (514263), \\
			& (215364), (214365), (415362), (514362)\}
\end{align*}
\end{exmpl}


\begin{thebibliography}{Dum74}

\bibitem[BD81]{dum-bar81}
Daniel Barsky and Dominique Dumont.
Congruences pour les nombres de {G}enocchi de 2e esp\`ece.
In {\em Study {G}roup on {U}ltrametric {A}nalysis. 7th--8th years:
  1979--1981 ({P}aris, 1979/1981) ({F}rench)}, pages Exp. No. 34, 13.
  Secr\'etariat Math., Paris, 1981.

\bibitem[DR94]{dum-rad94}
Dominique Dumont and Arthur Randrianarivony.
D\'erangements et nombres de {G}enocchi.
 {\em Discrete Math.}, 132(1-3):37--49, 1994.

\bibitem[Dum74]{dum74}
Dominique Dumont.
 Interpr\'etations combinatoires des nombres de {G}enocchi.
 {\em Duke Math. J.}, 41:305--318, 1974.

\end{thebibliography}
\end{document}